\documentclass[a4paper,12pt]{article}
\usepackage[top=2.5cm,bottom=2.5cm,left=2.5cm,right=2.5cm]{geometry}
\usepackage{cite, amsmath, amssymb}

\usepackage {float}
\usepackage {epsf}
\usepackage{algorithm}
\usepackage{algpseudocode}
\usepackage{graphicx}
\newtheorem{theorem}{Theorem}

\newtheorem{definition}{Definition}

\usepackage{color}

\newcommand{\modu}{\mbox{mod }}

\newenvironment{proof}%
{\begin{description}\item[Proof:]}
{\hfill$\blacksquare$
\end{description}}

\begin{document}

\begin{center}

{ \Large \bf \vspace*{3cm} ~\\ Goldberg, Fuller, Caspar, Klug and Coxeter \\ 
and a general approach to \\local symmetry-preserving operations }

\bigskip

{ \bf Gunnar Brinkmann$^1$,  Pieter Goetschalckx$^1$, Stan Schein$^2$ }

\bigskip 
{\em 
~$^1$Applied Mathematics, Computer Science and Statistics\\ Ghent University
Krijgslaan 281-S9\\ 9000 Ghent, Belgium
\medskip

~$^2$California Nanosystems Institute and Department of Psychology\\ University of California\\ Los Angeles, CA 90095

}
\bigskip

Gunnar.Brinkmann@UGent.be, Pieter.Goetschalckx@UGent.be, Stan.Schein@gmail.com. 

\bigskip

(Received ...)

 \end{center}






\section*{\underline{Abstract}}

Cubic polyhedra with icosahedral symmetry where all faces are
pentagons or hexagons have been studied in chemistry and biology as
well as mathematics. In chemistry one of these is
buckminsterfullerene, a pure carbon cage with maximal symmetry,
whereas in biology they describe the structure of {\em spherical}
viruses. Parameterized operations to construct all such polyhedra were
first described by Goldberg in 1937 in a mathematical context and
later by Caspar and Klug -- not knowing about Goldberg's work -- in
1962 in a biological context.  In the meantime Buckminster Fuller also
used subdivided icosahedral structures for the construction of his
geodesic domes.  In 1971 Coxeter published a survey article that
refers to these constructions.  Subsequently, the literature often
refers to the {\em Goldberg-Coxeter construction}. This construction is
actually that of Caspar and Klug. Moreover, there are essential
differences between this (Caspar/Klug/Coxeter) approach and the approaches of
Fuller and of Goldberg.
We will sketch the different approaches and generalize
Goldberg's approach to a systematic one encompassing all local
symmetry-preserving operations on polyhedra.

\bigskip

{\bf Keywords:} polyhedra, symmetry preserving, chamber system, \\Goldberg-Coxeter operation

\bigskip

{\bf Author's contributions:} All authors contributed critically in ideas and writing.\\

{\bf Data accessibility:} No data was involved in this research.\\

{\bf Competing interests:} None.\\

{\bf Acknowledgements:} The authors want to thank Nico Van Cleemput (Ghent University) for helpful discussions.\\

{\bf Funding statement:} There was no external funding for this project.\\

\bigskip

\section*{Introduction}

In mathematics and chemistry the article by Caspar and Klug \cite{CK62} is less popular than
the ones by Coxeter  \cite{Coxeter_virus} and Goldberg \cite{Goldberg_polyhedra}. 
A reason might be that the former has a strong biological focus and 
is vague with regard to mathematical details. Indeed, for details it refers to another paper
\cite{CK62-nonexist} that is listed as {\em to be submitted} in the article but was never completed, as 
Caspar writes in a later comment  \cite{caspar-klug-comment} on \cite{CK62}.

The most commonly used source for information about the Caspar-Klug
construction is the article by Coxeter \cite{Coxeter_virus} in which the construction
is described in a more formal way – and without referring much to the
biological context.
In this article Coxeter mentions the works of Caspar
and Klug, of Buckminster Fuller \cite{dymaxionworld} and of Goldberg, but when describing the
method of covering an icosahedron with identical equilateral triangles cut out of the triangular lattice
-- the construction often cited as the Goldberg-Coxeter operation (see e.g. \cite{GoldbergCox34}
or \cite{wrapping}) -- he explicitly refers to Caspar and Klug.

The misunderstanding that this operation is identical to the much
earlier proposed operation by Goldberg might have been caused by the sentence in \cite{Coxeter_virus}: 
{\em ``Independently of Michael Goldberg, Caspar and Klug proposed the
  following rule for making a suitable pattern.''} But, although the
2-parameter description of the operation is the same, the equilateral
triangle used by Caspar and Klug is not even an intermediate step in
Goldberg's approach and is not what Goldberg -- who in fact works in
the duals, the hexagonal lattice and the dodecahedron -- glues onto
the dodecahedron. The results are the same (or, to be exact, duals of
each other), but the methods to obtain these results differ in
essential points. 

We will now describe the methods of Goldberg
\cite{Goldberg_polyhedra}, of Fuller \cite{dymaxionworld}, and
of Caspar and Klug \cite{CK62} in the order in which they were developed. Finally we will 
show how Goldberg's approach can be
modified to form a general approach to local operations
preserving symmetry.

\section*{The approach of Goldberg}

We will describe the approach of Goldberg in more modern language than the one in his article
from 1937. This language allows a more formal description and will make generalization easier. 
In this article we discuss graphs that are finite or infinite, have a finite degree of 
the vertices, are embedded in the sphere or the Euclidean plane so that all faces have an 
interior that is homeomorphic to an open disc, and have a boundary for each face that is a 
simple cycle. 
We will refer to a plane graph (or polyhedron if the graph is 3-connected) if a 
finite graph is embedded in the Euclidean plane or on the sphere. We will refer to a tiling if 
it is an infinite tiling of the Euclidean plane. (Other tilings will not be discussed in detail in 
this article.)     
As shown in Figure~\ref{fig:chambers} for the example of the hexagonal tiling of the Euclidean 
plane, we can obtain a barycentric subdivision of such a tiling by labeling each vertex with a 
$0$, by inserting one vertex within each edge and labelling it with a $1$, and by placing one 
vertex in the interior of each face and labelling it with a $2$.  
Then we connect the vertices with label $2$ to all of the vertices ($0$'s and $1$'s) in the boundary of the face 
in which it lies in the cyclic order around the boundary. This procedure can also be used when 
the boundary of each face is not a simple cycle and vertices occur more than once in the walk around
the boundary. In this case the result would be a multigraph.

By this method we can obtain a tiling where all faces are triangles with vertices labeled (in clockwise order around the boundary)
either $0,1,2$ (we can call them black) or $2,1,0$ (white), and each triangle shares only edges with triangles 
of another colour. We call such labelled triangles {\em chambers}.
Let $\Sigma = \langle \sigma_0,\sigma_1,\sigma_2 | \sigma_0^2=\sigma_1^2=\sigma_2^2=1\rangle$
denote the free Coxeter group with generators $\sigma_0,\sigma_1,\sigma_2$.
Now, if we define for $0\le i\le 2$ and two chambers $t,t'$ such that 
$t$ and $t'$ share an edge but do not share the vertex labelled $i$ that
$\sigma_i(t)=t'$, we have a chamber system.
In an abstract way, these chambers can also be defined as elements of the flag space of the polyhedron.
The flags are given by triples $(v,e,f)$, so that vertex $v$ is contained in edge $e$ and edge $e$ is contained in face $f$.
See \cite{DrBr96} and \cite{DressHuson87} for more detailed descriptions. The chamber system 
encodes the combinatorial
structure of the polyhedron completely, and the symmetry group of the polyhedron induces an
operation on the chamber system.

\begin{figure}
\begin{center}

\includegraphics[width=9cm]{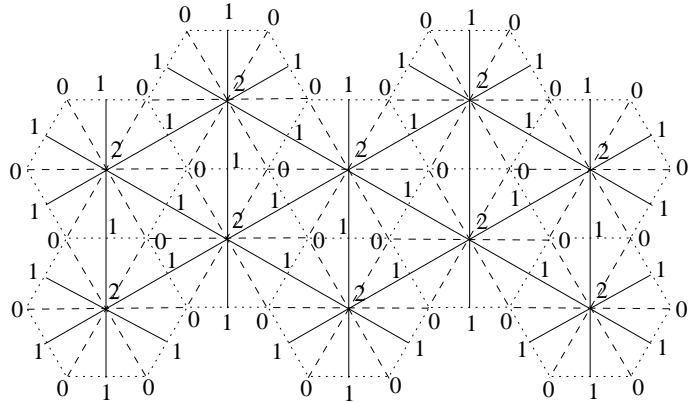}

\end{center}
\caption{The chamber system of the hexagonal tiling. Here and in the remainder of this article
dotted lines correspond to edges of the original tiling (not containing the vertex labelled $2$), 
full lines to connections between face centers and edge centers (not containing vertex $0$)
and dashed lines to connections between face centers and vertices (not containing vertex $1$). }\label{fig:chambers}
\end{figure}

While Caspar and Klug cut out triangular parts that fill complete faces of the triangulation
(e.g. of the icosahedron)
onto which they are glued, Goldberg proposes to cut out smaller triangles and 
to fill an $f$-gon with $2f$ copies of small triangles. Although he applies the method only to
pentagons, different from the Caspar-Klug method, his method can be applied to faces of any size.

Unfortunately Goldberg's article \cite{Goldberg_polyhedra} has a mistake. The construction only works for
achiral icosahedral fullerenes, which have the full icosahedral group of order $120$, and not -- as he claims -- also for 
chiral fullerenes with only the rotational subgroup of order $60$. He argues that for fullerenes
(or {\em medial polyhedra} as he calls them) that are
assembled from $12$ congruent {\em patches} that are pentagonal in shape we have that 
{\em ``Each pentagonal patch may be divided into ten equivalent (congruent or symmetric) triangular patches.''}
This statement is true for achiral icosahedral fullerenes but not for pentagonal patches of chiral icosahedral fullerenes.
The latter can be divided into five congruent triangular patches or ten triangular patches consisting of
two sets of five congruent triangles. The error is surprising because in his article he gives a picture 
(Figure~3 in \cite{Goldberg_polyhedra}) of the triangular
patch he claims to be sufficient for the chiral parameters $(5,3)$, and copies of this triangular patch
can obviously not be identified along edges to form a tiling.

Goldberg's method for cutting the patches out of the hexagonal lattice can thus 
(in a more formal and corrected way) be described as follows:
Assume that the regular hexagonal lattice is equipped with a coordinate system so that
$(0,0)$ is in the center of a hexagonal face $f$, $(1,0)$ is a vertex $v$ of $f$ and $(0,1)$ is
the vertex of $f$ that is obtained by a $60$ degrees rotation of $v$ in counterclockwise
direction around the origin.  By induction 
it can be easily proven that a vertex with integer coordinates $(a,b)$ is the center of a face
if $a-b \equiv 0 (\modu 3) $ and a vertex otherwise. Or, to be more precise, if $a-b \not\equiv 0 (\modu 3) $
each vertex is either
the left or the right vertex of an edge parallel to the vector $(1,0)$. If $a-b \equiv 1 (\modu 3) $,
it is the left vertex. If $a-b \equiv 2 (\modu 3)$, it is a right vertex.
If $(a,b)$ are not both integer but both
multiples of $\frac{1}{2}$, then $(a,b)$ is the center of an edge if $2(a-b)  \equiv 0 (\modu 3)$.

For given parameters $l,m\in  \mathbb{N}$ with $l\ge m\ge 0$, choose $v_2=(0,0)$,
$v_1=(\frac{l-m}{2},\frac{l+2m}{2})$ and $v_0=(l,m)$. The point $v_1$ is the midpoint
between $v_0$ and the image of $v_0$ when rotating it by $60$ degrees 
in counter clockwise direction around the origin $v_2$. See Figure~\ref{fig:70} for an example.
The resulting triangle is always a triangle with a $30$ degree angle at $v_2$, a $90$ degree
angle at $v_1$ and a $60$ degree angle at $v_0$. While $v_2$ is always the center of a face,
$v_0$ can be a vertex or the center of a face. For $v_1$ we have that 
$\frac{l-m}{2}-\frac{l+2m}{2}=\frac{-3m}{2}$. Both coordinates of $v_1$ are only integer if
$l$ and $m$ are both even. In this case $v_1$ is the center of a face. 
Otherwise $2\frac{-3m}{2}=-3m\equiv 0(\modu 3)$, so in that case $v_1$  is the center of an edge.

\begin{figure}
\begin{center}

\includegraphics[width=7cm]{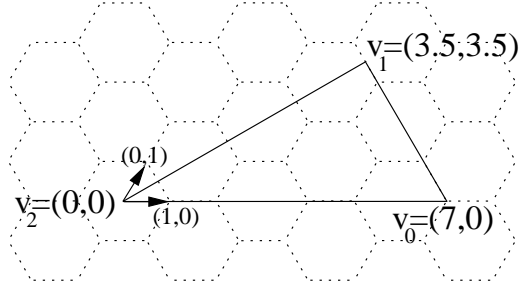}

\end{center}
\caption{The triangle used by Goldberg for the parameters $(7,0)$. }\label{fig:70}
\end{figure}

Although he used different words, Goldberg proposed to glue a copy of the 
interior of this triangle $v_0, v_1, v_2$, which we will call a {\em Goldberg triangle}, 
into each chamber 
of a polyhedron in such a way that the vertices $v_0, v_1, v_2$ are identified with the 
corresponding vertices of the chambers.
This construction
glues each side of a triangle against the same side in the mirror
image of a copy of the triangle. If the sides in the hexagonal lattice
are part of a mirror axis, the edges nicely match to form a tiling of
the sphere because they match in the hexagonal tiling. If they are
not, lines do not necessarily match.  
The side $v_2, v_0$ has direction $(l,m)$ which is part of a mirror axis through the origin
for $l ≥ m$ if and only if $m = 0$ or $l = m$.
If $m=0$ the side $v_2,v_1$ has
direction $(\frac{l}{2},\frac{l}{2})$ and starts at a face center, and
the side $v_1,v_0$ has direction $(\frac{l}{2},-\frac{l}{2})$ starting
at a face center ($l$ even) or the center of an edge parallel to this
direction.  If $l=m$, the side $v_2,v_1$ has direction
$(0,\frac{3l}{2})$ starting at a face center, and side $v_0,v_1$ has
direction $(-l,\frac{l}{2})$ also starting at a face center.  In each
case, all three sides of the triangle are mirror axes. When applied to
a polyhedron or tiling, new faces not containing a corner point of the
subdivided chamber are hexagons as in the lattice, and vertices not
forming a corner point remain 3-valent. 
At $v_1$ four copies are glued together in the same way as in the hexagonal lattice, so 
the local situation is exactly like that in the hexagonal lattice. 
As the dodecahedron is
3-valent, at $v_0$ six copies are glued together, just as they are
situated in the hexagonal lattice, so this gluing process will give a 3-valent vertex
if $v_0$ is a vertex and a hexagon if $v_0$ is the center of a face.
Only at $v_2$ can face sizes differ from $6$: In the resulting polyhedron, the vertex $v_2$ will
be the center of a face of the same size as before, thus a pentagon in
the case of decorating the dodecahedron.

If $l\not= m$ and $m>0$, the result of the gluing process can give disconnected graphs or double edges, because the
triangles are glued to parts that are different from the parts they are glued to in the hexagonal lattice.
For $l\not= m$ and $m>0$, the sides are not mirror axes, so the content of the congruent triangles on
the other side of edges is different (Figure~\ref{fig:53}). 
To this end, in addition to $v_0,v_1$, and $v_2$ another vertex $v'_0$ with coordinates $(-b,a+b)$ must be used. 
We mark the triangle $v'_0,v_2,v_1$ black, whereas we mark the triangle $v_0,v_1,v_2$ white.
Gluing copies of the black triangles to all black triangles of the chamber system of
the dodecahedron and copies of the white triangles to all white triangles of the chamber system of
the dodecahedron produces the desired icosahedral fullerene. The mirror symmetries of the dodecahedron
are lost, as the content of black and white triangles cannot be mapped onto each other by a reflection,
but all orientation-preserving symmetries remain. The facts that $v_1$ is a center of a 2-fold rotation,
that $v_2$ is a center of a 6-fold rotation, and that $v_0$ and $v'_0$ are centers of 3-fold rotations ensures
that the triangles match up. As each $n$-gon contains $n$ black and $n$ white triangles that can be paired,
it is not necessary to distinguish between the white and the black triangles, but the larger quadrangle $v_2,v_0,v_1,v'_0$
(that has in fact a triangular shape due to the angle of $180$ degrees at $v_1$), 
containing a black triangle and a white triangle, can be used to decorate the paired chambers in the polyhedron.

\begin{figure}
\begin{center}

\includegraphics[width=7cm]{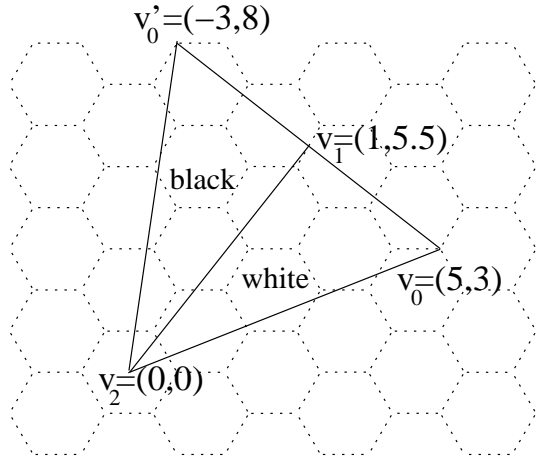}

\end{center}
\caption{The triangles Goldberg should have used for the parameters $(5,3)$. In his article he only uses
the triangle $v_0,v_1,v_2$.}\label{fig:53}
\end{figure}

For a more formal way to describe the combinatorial structure of the resulting fullerene or for
implementing the operations in a computer program, it is better
to switch to chamber systems completely. The symmetry group operates on the chambers of the
hexagonal tiling, and the barycentric subdivision 
can be chosen in a way that complete chambers are contained 
in the Goldberg triangle. This way the
gluing of the Goldberg triangle comes down to subdividing a chamber of the dodecahedron
into smaller chambers. In the chiral case the set of chambers
of the hexagonal tiling can be chosen in various ways as fundamental domains of the group generated
by the rotations required for the sides of the triangles. 
For a given barycentric subdivision it is possible that when drawing the triangle with straight lines chambers are 
intersected and only partially contained in the triangle. 
The set of chambers must be chosen in a way that the boundary is a simple cycle
through $v_0,v_1,v'_0,v_2$, so that the path from $v_1$ to $v_0$ is the image of  the path from $v_1$ to $v'_0$
under a rotation of $180$ degrees around $v_1$ and that  the path from $v_2$ to $v'_0$ is the image of  
the path from $v_2$ to $v_0$ under a rotation of $60$ degrees around $v_2$.
Figure~\ref{fig:53chamber}
gives an example for parameters $(5,3)$.

Details on operations not necessarily preserving orientation reversing symmetries will be given in
a later article. 

\begin{figure}
\begin{center}

\includegraphics[width=7cm]{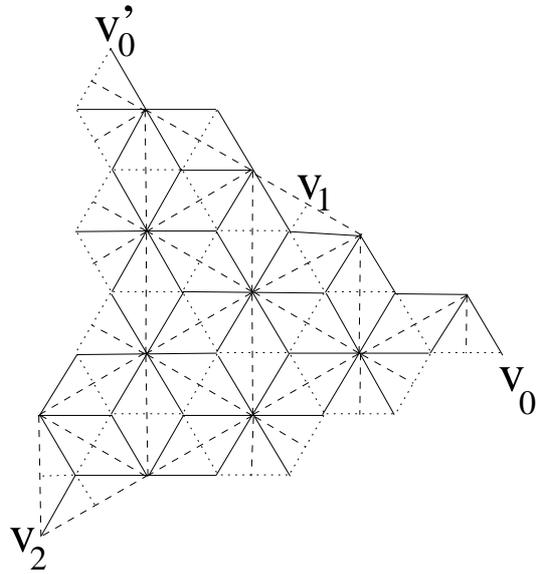}

\end{center}
\caption{The two triangles for the parameters $(5,3)$ described as chambers.
As before, dotted lines correspond to edges of the original tiling (not containing the vertex labelled $2$), 
full lines to connections between face centers and edge centers (not containing vertex $0$)
and dashed lines to connections between face centers and vertices (not containing vertex $1$). }\label{fig:53chamber}
\end{figure}

\section*{The approach of Buckminster Fuller}

The first geodesic dome is believed to be the provisional dome of
the Zeiss-Planetarium on the roof of the Zeiss factory in Jena around 1924. 
The structure was designed by Walter Bauersfeld, an optical engineer who worked for
Zeiss, and was patented and constructed by the firm of Dykerhoff and Wydmann.
It was already based on a subdivided icosahedron in order to get close to a round surface
suitable for the projection of the stars.

Fuller reinvented this idea about 20 years later and popularized it in the United States.
Nowadays it is mainly his name that is associated with geodesic domes.
Refinements of the icosahedron and the cuboctahedron that he called {\em vector equilibrium} or 
{\em Dymaxion} (see Figure~\ref{fig:dymaxion}) were also used for the early version of his
{\em Dymaxion map} -- a map of the world, projected onto the surface of an inscribed
polyhedron. 

\begin{figure}
\begin{center}

\includegraphics[width=3cm]{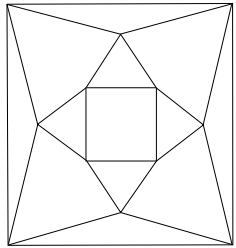}

\end{center}
\caption{A Schlegel diagram of the Dymaxion polyhedron (the cuboctahedron) used by Buckminster Fuller. }\label{fig:dymaxion}
\end{figure}

As his work preceded that of Caspar and Klug, it is of course interesting to know whether for his refinements
of the icosahedral and Dymaxion structures he had already used the methods of Caspar and Klug based on the 
corresponding Euclidean lattices.
In \cite{dymaxionworld}(p. 50) Fuller and Marks write as follows:
\medskip

{\em ``The geometric principles underlying the Dymaxion map are the same as those used to develop
the basic pattern of Fuller's domes.''} 
\medskip

and
\medskip

{\em ``To produce Fuller's Dymaxion map, we reverse this process. We start with a sphere, on whose surface a
spherical icosahedron has been drawn. Next we subtriangulate the icosahedron's 20 triangular faces
with symmetric, three way, great circle grids of a chosen frequency. Then we transfer this figure's configuration
of points to the faces of an ordinary (non-spherical) icosahedron which has been symmetrically subtriangulated
in frequency of modular subdivision corresponding to the frequency of the spherical icosahedron's subdivisions.''} 
\medskip

So Fuller's construction was less general. He {\em subtriangulated} the triangles without 
referring to the triangular lattice and without the
possibility to produce chiral structures where some smaller triangles cross the edges of the original
icosahedron.

\section*{The approach of Caspar and Klug}

Caspar and Klug use the Euclidean triangular lattice equipped with a
coordinate system, so that -- with the origin in a vertex -- the
points $(0,1)$ and $(1,0)$ are on vertices neighbouring the origin,
and $(0,1)$ can be obtained from $(1,0)$ by a $60$ degrees rotation in
counterclockwise direction around the origin. In this lattice the
vertices are exactly the points with both coordinates integers
(Figure~\ref{fig:CasparKlug}).  Note that when we use the coordinate
system of Goldberg that is defined in the dual hexagonal lattice
(Figures~\ref{fig:70} and \ref{fig:53}), the length of the unit
vectors used by Goldberg is shorter by a factor of $\sqrt{3}$ than that of the
vectors used by Caspar and Klug.

\begin{figure}
\begin{center}

\includegraphics[width=7cm]{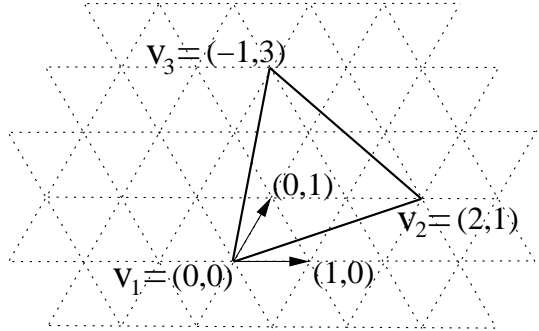}

\end{center}
\caption{The triangle used by Caspar and Klug for the parameters $(2,1)$. }\label{fig:CasparKlug}
\end{figure}

Caspar and Klug propose to look at the triangle formed by the vertices $v_1=(0,0)$, $v_2=(a,b)$ and $v_3=(-b,a+b)$
for integer values $a\ge b \ge 0$. This triangle has all corners in vertices, and vertices of the triangular lattice
are centers of rotational symmetries of the tiling by $60$ degrees. The product of two rotations by $60$ degrees in clockwise
order -- first around $v_1$, then around $v_2$ -- is a rotation
by $120$ degrees around the center of the
chosen triangle, proving that this rotation is a symmetry of the triangle, its interior, and the tiling.

These observations show that however we join identical copies of this Caspar-Klug triangle 
by identifying the copies pairwise along a side 
to form a closed structure, 
all faces will be triangles and all vertices that are 
not a corner of a Caspar-Klug triangle will have degree $6$. If the triangles
are glued into a triangulation of some surface, the degrees of the vertices of the original triangulation
stay the same.

Automorphisms of a triangulation of the sphere map triangles to triangles. As the Caspar-Klug triangles always
have the full set of orientation-preserving symmetries of the triangle, all orientation-preserving symmetries
can be extended to the interior of the triangles. In other words: each orientation-preserving automorphism of the
triangulation is also an automorphism of the decorated triangulation. 

If (and only if) the parameters $(a,b)$ fulfill $b=0$ or $a=b$, the Caspar-Klug triangles also have mirror
axes perpendicular to the midpoints of the edges. In this case the Caspar-Klug triangles have the full symmetry group
of the triangle, and all automorphisms of a triangulation of the sphere that is decorated can be extended
to the decorated triangulation, so that the whole automorphism group is preserved.

\section*{A general approach to local symmetry-preserving \\ operations on polyhedra}

The use of operations on polyhedra possibly dates back to 
the ancient Greeks, who were the first to describe
the Platonic solids and the Archimedean solids that can be constructed
from the Platonic solids by simple operations that preserve the
symmetry of the original object. 
When he rediscovered the Archimedean solids in
his book {\em Harmonices Mundi} \cite{Kepler}, Johannes Kepler
coined the names now used for the Archimedean solids. The name
{\em truncated octahedron} shows clearly that he considered 
this polyhedron to be constructed by truncation -- a symmetry-preserving
operation -- from the octahedron.

Although these operations are often used and 
well-studied in mathematics, there currently exists no
systematic way to describe them. There exists an extensive naming scheme
using terms like {\em ambo}, {\em kis}, {\em truncate}, {\em cantellate},
{\em runcinate}\ etc., and for several subclasses there exist
different techniques to describe them, e.g. the Conway polyhedron notation\cite[Chapter
  21]{CBGS:2008}, Schl\"afli symbols \cite{Coxeter73} or   Wythoff symbols\cite{CLM:1954}.
Nevertheless there is no definition of {\em local symmetry-preserving operations} and therefore
also no technique to systematically describe all possible
local symmetry-preserving operations or theorems about what 
local symmetry-preserving operations can do or cannot do. A list of popular operations
can also be found on the Wikipedia site for the Conway notation \cite{conway_notation}.

In addition to these mathematically motivated operations that produce polyhedra from polyhedra, 
there is a long tradition in art and design of decorating polyhedra and other 
objects with parts of periodic tilings. Most famous are probably the {\em Sphere with fish}
and the {\em Sphere with Angels and Devils} by M.C. Escher. For more information 
and examples
see \cite{escherbox} and \cite{decpolyhedra}.

While operations on polyhedra were mainly interpreted as manipulating the solid object
(e.g. truncation as cutting off vertices and thereby producing new faces),
the results of the decorations were interpreted as decorated polyhedra -- and not 
just new and different polyhedra. Only by interpreting decorations
as combinatorial operations does it become clear how closely these two approaches
are connected.

The term {\em symmetry-preserving} can easily be transformed into an
exact requirement: the symmetry group of the polyhedron to which the
operation is applied must be a subgroup of the symmetry group of the
polyhedron that is the result of the operation. In most cases it will
be the whole symmetry group of the result.  This goal could also be
obtained by taking global properties into account -- e.g. the symmetry
group itself -- so the terms {\em local} and {\em operation} are still to be
made precise. The classical operations like {\em truncation} are
obviously what must be included in the definition, so checking that
all classic operations are covered by the definition is a first test of the usefulness of the
following definition.

\begin{definition}\label{def:chamber}

Let $T$ be a periodic 3-connected tiling of the Euclidean plane with chamber system $C_T$
that is given by a barycentric subdivision that is invariant under the symmetries of $T$.
Let $v_0,v_1,v_2$ be points in the Euclidean plane
so that for $0\le i < j \le 2$ the line $L_{i,j}$ through $v_i$ and $v_j$ is a mirror axis of the tiling.

If the angle between $L_{0,1}$ and $L_{2,1}$ is $90$ degrees, the angle between 
$L_{2,1}$ and $L_{2,0}$ is $30$ degrees and consequently the angle between 
$L_{0,1}$ and $L_{0,2}$ is $60$ degrees, then the triangle $v_0,v_1,v_2$ 
subdivided into chambers as given by $C_T$ and
the corners $v_0,v_1,v_2$ labelled with their names $v_0,v_1,v_2$ is called 
a local symmetry-preserving operation, {\em lsp operation} for short.

The result of applying an lsp operation $O$ to a tiling or polyhedron $P$ is
given by subdividing each chamber $C$ of the chamber system $C_P$ of $P$ with $O$ by
identifying for $0\le i\le 2$ the vertices of $O$ labelled $v_i$ with the vertices labelled $i$ 
in $C$. Note that this operation is purely combinatorial and that the
angles of $O$ are not preserved.

\end{definition}

The choice of $30$ degrees, resp. $60$ degrees is a reference to
Goldberg's paper.  For people familiar with tilings and Delaney-Dress
symbols as in \cite{DressHuson87}, it is immediately clear that one could also
interchange the values, require $45$ degrees for both or even choose
other values and go to hyperbolic or spherical tilings and still get
the same set of operations.  For example, if one would use $45$
degrees for both angles, the dual would come from the regular square
tiling but would be the same combinatorial operation. As another
example Figure~\ref{fig:4545truncate} shows the tiling from which
truncation can be obtained with two angles of $45$ degrees.

An alternative way to understand that different choices of angles give the same set of operations is as follows:
Let $T_0$ be a tiling of the sphere or the Euclidean or hyperbolic plane where the symmetry group acts 
transitively on the chambers. Then any lsp operation as defined in Definition~\ref{def:chamber}
can be applied to $T_0$ to give another tiling $T_1$
from which the operation can be recovered by requiring the set of angles as in the chambers of $T_0$.

\begin{figure}
\begin{center}

\includegraphics[width=7cm]{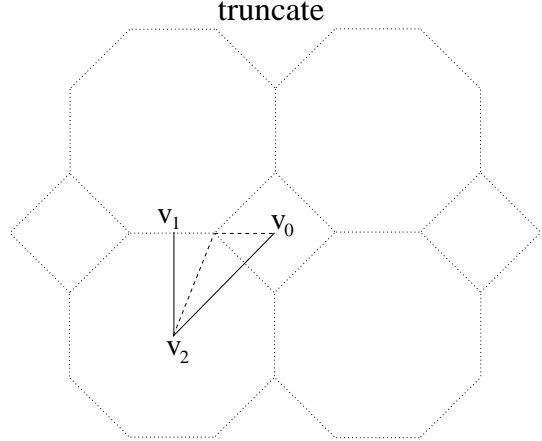}

\end{center}
\caption{ The operation {\em truncation} for two angles of $45$ degrees 
instead of $30$ degrees and $60$ degrees. }\label{fig:4545truncate}
\end{figure}

The symmetry group of the tiling $P$ is an automorphism of the chamber
system $C_P$.  This implies immediately that it is also an
automorphism of the result $O(P)$ justifying the term {\em symmetry-preserving operation}. The fact that operations are defined on the
level of chambers justifies the term {\em local}. Nevertheless one
must ask whether all operations one wants to call {\em local symmetry-preserving operation} are covered. We motivate our definition by
showing that all well known and often used operations on polyhedra are
covered by our definition.  In Figure~\ref{fig:ops} we give the
chamber operations for the operations {\em identity}, {\em dual}, 
{\em  ambo}, {\em truncate}, {\em chamfer}, and {\em quinto}.  Other operations (like
    {\em join}, {\em kis}, {\em expand}, {\em otho}, {\em bevel}, {\em meta}, etc.) can be written as
    products of these operations.


\begin{figure}
\begin{minipage}{5cm}
\begin{center}
\includegraphics[width=5cm]{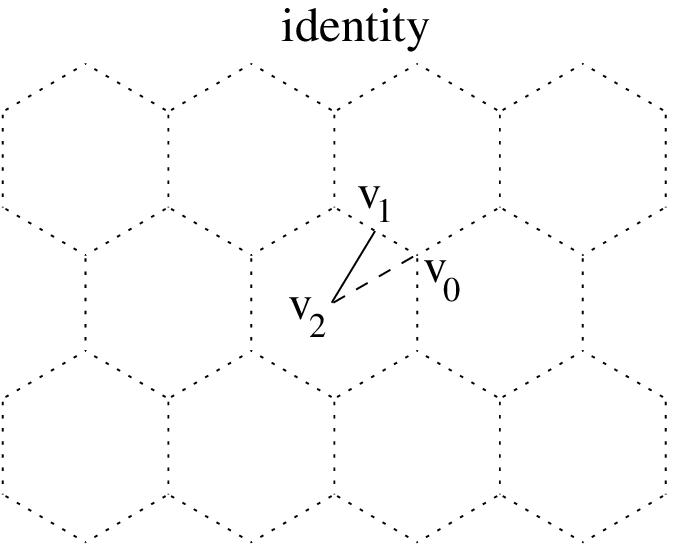}
\end{center}
\end{minipage} \hfill
\begin{minipage}{6cm}
\begin{center}
\includegraphics[width=6cm]{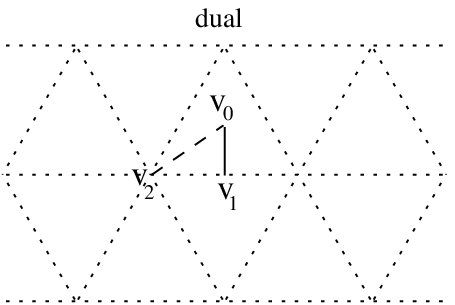}
\end{center}
\end{minipage} 

\begin{minipage}{7.5cm}
\begin{center}
\includegraphics[width=7.5cm]{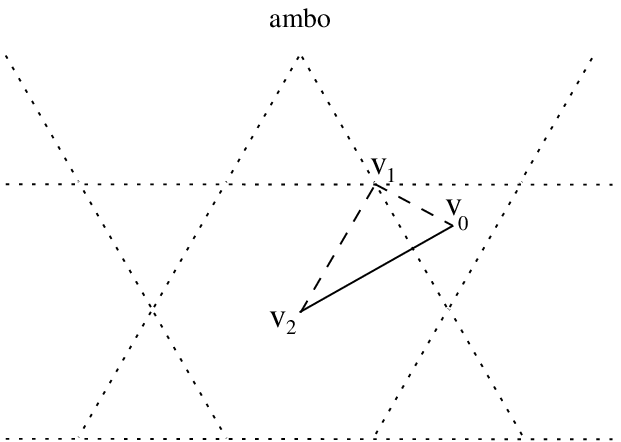}
\end{center}
\end{minipage}  \hfill
\begin{minipage}{8cm}
\begin{center}
\includegraphics[width=7.0cm]{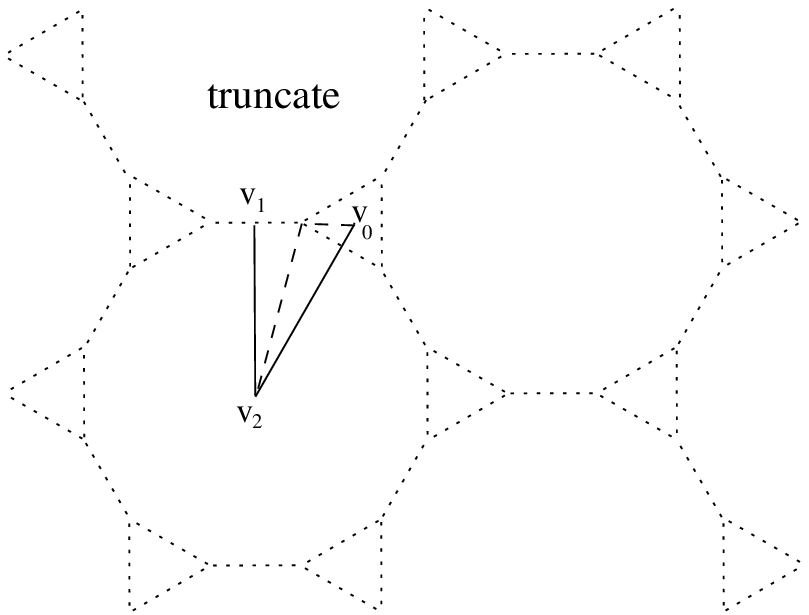}
\end{center}
\end{minipage} 

\begin{minipage}{7cm}
\begin{center}
\includegraphics[width=7cm]{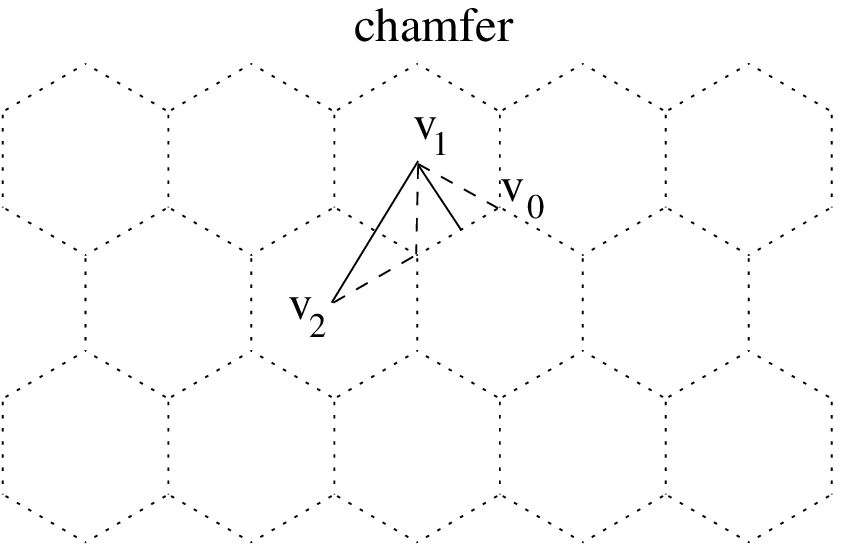}
\end{center}
\end{minipage}   \hfill
\begin{minipage}{8cm}
\begin{center}
\includegraphics[width=7.5cm]{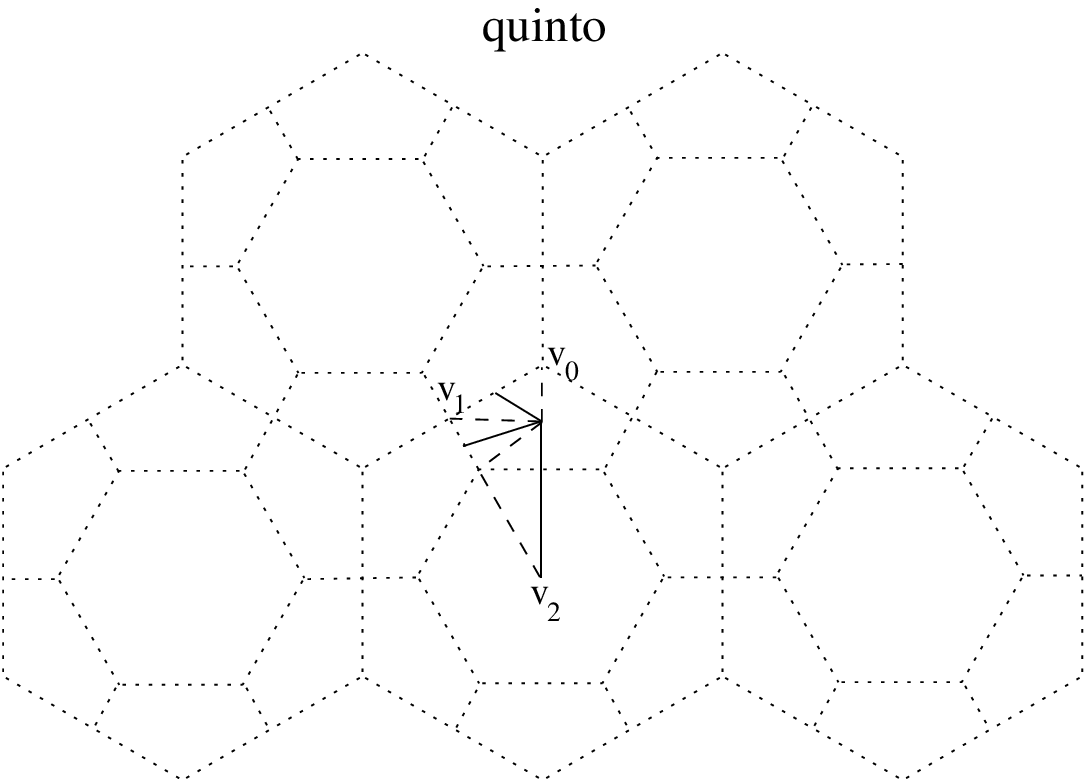}
\end{center}
\end{minipage} 
\caption{ The operations identity, dual, ambo, truncate, chamfer, and
  quinto. For better visibility of the operations, the chambers are
  only drawn inside the triangular region defining the
  operation.}\label{fig:ops}
\end{figure}

Though initially described as purely geometrical operations, lsp operations have 
also been studied combinatorially 
(see e.g. \cite{RiosFrancos2014}) and have been described as subdivisions of chambers
of the  tiling or polyhedron.
Figure~7 in \cite{RiosFrancos2014} describes how the chamfering operation can be 
implemented as a chamber operation and can be applied to a map
given as an abstract chamber system. Representing operations on  polyhedra as chamber operations means that
the operation is given and coded or described as a chamber operation, so the operation
defines the decomposition of the chamber. Definition~\ref{def:chamber} establishes
the reverse direction: it says which decompositions of chambers into smaller chambers define
an operation.

This characterization is necessary as 
not every subdivision of a chamber defines an operation that transforms a chamber system into another
chamber system of a tiling or polyhedron -- see Figure~\ref{fig:noop} for an easy example. One of the properties
of chamber systems of tilings or polyhedra is that each $1$-vertex 
is contained in exactly $4$ chambers. Applying the subdivision in Figure~\ref{fig:noop} to a  tiling or
polyhedron, the result would not have this property.

\begin{figure}
\begin{center}

\includegraphics[width=2.5cm]{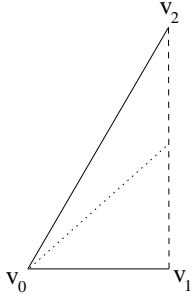}

\end{center}
\caption{ A subdivision of a chamber that does not define an lsp operation. }\label{fig:noop}
\end{figure}

While in \cite{CK62} and \cite{Goldberg_polyhedra} it is simply
assumed that the result of applying the construction to a dodecahedron
or an icosahedron is again a polyhedron, this assumption should be
proven in general. In order to justify our approach to lsp 
operations we have to prove that after applying such an operation to a polyhedron,
the result is again a polyhedron. 
The theorem of Steinitz 
allows translation of this requirement to a purely combinatorial theorem:

\begin{theorem}\label{thm:3con}
If $P$ is a polyhedron and $O$  an lsp operation,
then $O(P)$ is a polyhedron.
\end{theorem}

\begin{proof}

The fact that the resulting graph is plane is obvious. The fact that it is 3-connected
needs some more work. Let $G$ be a plane graph or tiling 
and $B_G$ the barycentric
subdivision of $G$.

It is well known and easy to prove that $G$ is $3$-connected if and only if $B_G$
has the following properties:

\begin{description}

\item[(i)] $B_G$ is a simple graph -- that is: there are no cycles of length $2$.

\item[(ii)]  The only cycles in $B_G$ of length $3$ are boundaries of chambers.

\item[(iii)]  The only cycles in $B_G$ of length $4$ have on one side two triangles sharing an edge
or $4$ triangles sharing a 1-vertex.

\end{description}

As the barycentric subdivision $B_{O(P)}$ is a subdivision of $B_P$, for each edge $e$
of $B_{O(P)}$ there is a chamber $C(e)$ of $B_P$ containing it. If the edge is included in
the boundary of a chamber, $C(e)$ is not unique. In order to prove that 
$O(P)$ is 3-connected, we have to prove that $B_{O(P)}$ has properties (i), (ii), and
(iii).

Figures~\ref{fig:for_proof} and \ref{fig:for_proof2} give
all possibilities for a 4-cycle in $B_P$ together with all chambers containing edges of the
cycles. If vertices displayed in these figures as distinct vertices would represent
the same vertex in $B_P$, this would contradict at least one of the properties
(i), (ii), and (iii), so all vertices shown are pairwise distinct.
Figures~\ref{fig:for_proof} and \ref{fig:for_proof2} also show corresponding
areas in the tiling for an lsp operation. From the definition of an operation (that is: the requirement
for the location of mirror planes), it follows that all vertices in the displayed area are pairwise distinct.
After the decoration, there is an isomorphism from the vertices of the barycentric subdivision $B_{O(P)}$
inside the drawn area onto the vertices of the barycentric subdivision of the tiling in the drawn area.

\begin{figure}
\begin{center}

\includegraphics[width=8.5cm]{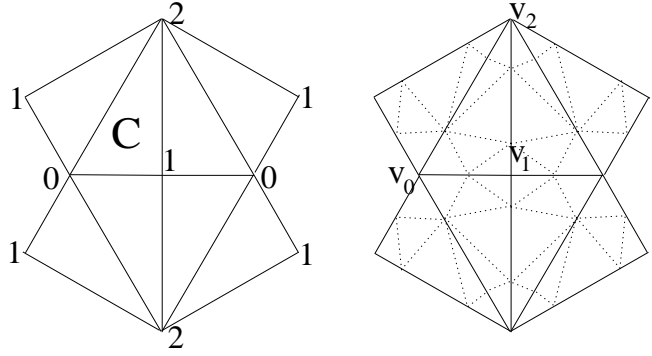}

\end{center}
\caption{ The first figure shows a 4-cycle around 4 chambers sharing a 1-vertex and all chambers sharing an edge of the cycle.
The figure also contains 4-cycles around two chambers intersecting in a 1-vertex and a 0-vertex or
intersecting in a 1-vertex and a 2-vertex. The second figure shows the corresponding
area in the tiling for an lsp operation. In the area for the lsp operation an example
tiling is drawn.}\label{fig:for_proof}
\end{figure}

\begin{figure}
\begin{center}

\includegraphics[width=8.5cm]{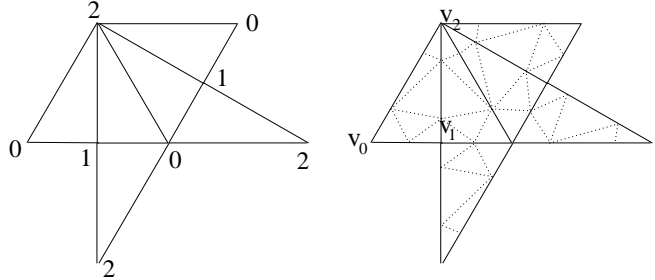}

\end{center}
\caption{ The first figure shows a 4-cycle around two chambers sharing the 0-vertex and the 2-vertex 
and all chambers sharing an edge of the cycle. The second figure shows the corresponding
area in the tiling for an lsp operation. In the area for the  lsp operation an example
tiling is drawn.}\label{fig:for_proof2}
\end{figure}

For a cycle $e_0,\dots ,e_k$ in $B_{O(P)}$ let $C(e_0), \dots ,C(e_k)=C_0,\dots ,C_k$ be the cyclic sequence of chambers 
of $B_P$, so that for $0\le i\le k$ edge $e_i$ is contained in $C(e_i)$. 
As long as the sequence of chambers has more than one chamber and we have
$C_i=C_{i+1}$ for some $i$ with the indices taken modulo $k+1$,
we can remove $C_i$ (and renumber the chambers). This way we get
the reduced sequence $C_0, \dots ,C_m$. Let now $e_0,\dots ,e_k$ be a cycle in $B_{O(P)}$
with length at most $4$. The  reduced sequence $C_0, \dots ,C_m$ also has length $m+1\le 4$.
If $m+1\le 2$, the cycle is completely contained in one chamber or in two chambers sharing an edge.
As the image of $e_0,\dots ,e_k$ under $\phi$ is a cycle of the same length, it must satisfy
the requirements described in (i),(ii), and (iii), so also $e_0,\dots ,e_k$ must satisfy these
requirements.

Assume now that $m+1\ge 3$. Depending on whether $e_0,\dots ,e_k$ crosses the border
between two chambers in a vertex or an edge, for $0\le i \le m$ the chambers $C_i$ and
$C_{i+1}$ (with the indices modulo $m+1$) share one or two vertices. Furthermore
each $C_i$ shares in total at least two vertices with the others. Checking the different
configurations of $3$ or $4$ chambers in a chambersystem that have these properties
one can prove that  there is a cycle $Z$ of length $m+1\ge 3$
in $B_P$ with an edge in each of $C_0, \dots ,C_m$.


If $m=2$, $Z$ is 
the boundary of a chamber in $B_P$. This implies that the cycle $e_0,\dots ,e_k$
lies completely in the region described by a chamber $C$ of $B_P$ and the
chambers sharing an edge with $C$, so that again the cycle together with its
interior can be mapped into the tiling by an isomorphism.

If $m=3$, $Z$ must either be the boundary of the union
of $2$  chambers sharing an edge or the boundary of
the union of $4$  chambers sharing a 1-vertex. 
In each case we can map the interior of $Z$ and the decoration of chambers neighbouring $Z$
isomorphically into the tiling, which implies that $e_0,\dots ,e_4$ satisfies property (iii).

Together these observations imply that $B_{O(P)}$ fulfills (i),(ii), and (iii)
and that $O(P)$ is 3-connected.

\end{proof}

\subsection*{\underline{A first step towards general local operations} \\ \underline{preserving orientation-preserving symmetries}}

Although the emphasis of this paper is on the work of Goldberg, Fuller, Caspar and Klug and on operations preserving all
symmetries, we want to make a first step -- details are postponed to a later paper -- 
towards operations that preserve orientation-preserving automorphisms
but not necessarily orientation-reversing automorphisms.
In the case of Goldberg, they could be described
by two subdivisions -- one of a black and one of a white triangle. In general, this is not possible, and it is necessary
to describe how a quadrangle containing a white and a black triangle that share an edge must be subdivided.

\begin{definition}\label{def:doublec}

Let $T$ be a periodic 3-connected tiling of the Euclidean plane 
with chamber system $C_T$,
and let $v_2,v_0,v_1,v'_0$ be points in the Euclidean plane so that:

\begin{itemize}

\item $v_2$ is the center of a rotation $\rho_{v_2}$ by $60$ degrees in counterclockwise
direction that is a symmetry of the tiling.

\item $v'_0=\rho_{v_2}(v_0)$ 

\item $v_1$ is the midpoint between $v'_0$ and $v_0$ and the center of
  a rotation $\rho_{v_1}$ by $180$ degrees that is a symmetry of the
  tiling (so also $v'_0=\rho_{v_1}(v_0)$).

\end{itemize}

Let $Q$ be a simple cycle in $C_T$ through $v_2,v_0,v_1,v'_0$ (in this order).
For $\{x,y\}\in \{\{v_2,v_0\},\{v_2,v'_0\},\{v_1,v'_0\},\{v_1,v_0\}\}$ 
let $P_{x,y}$ denote the path on $Q$ from 
$x$ to $y$ not containing any other vertex of  $\{v_2,v_0,v_1,v_0'\}$.

If $P_{v_2,v'_0}= \rho_{v_2}(P_{v_2,v_0})$ and $P_{v_1,v'_0}= \rho_{v_1}(P_{v_1,v_0})$,
then we call the quadrangle $Q$ with cornerpoints $v_2,v_0,v_1,v'_0$
labelled with the names $v_2,v_0,v_1,v_0'$
and subdivided  into chambers as given by $C_T$, 
a local operation that preserves orientation-preserving symmetries, {\em lopsp operation} for short.

Let $P$ be a polyhedron or tiling given as a chamber system. We call a pair of (black and white) chambers
sharing a $1$-vertex and a $2$-vertex a double chamber. Each chamber is contained in exactly one
double chamber.
The result of applying a lopsp operation $O$ to a  polyhedron or tiling $P$ is
given by subdividing each double chamber of $P$ with $O$ by
identifying $v_2$ with the $2$-vertex of the double chamber, $v_0$ with the $0$-vertex of the white chamber,
$v_1$ with the $1$-vertex of  the double chamber,
and $v'_0$ with the $0$-vertex of the black chamber.
Note that this operation is purely combinatorial and that the
angles of $O$ are not preserved.

\end{definition}

As already mentioned in the section about Goldberg's approach where the chiral case is discussed,
also here 
there are various ways to draw the boundaries of the quadrangle $Q$ that lead to 
equivalent operations, that is, operations that when applied to a polyhedron give the same 
result. We will postpone working out the details to a later paper.

It is tempting to conjecture that for each lopsp operation there is a way to choose the quadrangle
so that the subdivision of the double chamber can be expressed by two subdivisions of the chambers
-- one for the black chambers and one for the white chambers. Figure~\ref{fig:counterexample}
gives an example of an operation where this is not the case. In order to avoid misunderstandings,
we mention the fact already here, but a proof has to be postponed as first the necessary
prerequisites (e.g. results about different ways to choose the edges of the quadrangles
that lead to equivalent operations) have to be formally introduced and proven.

\begin{figure}
\begin{center}

\includegraphics[width=12cm]{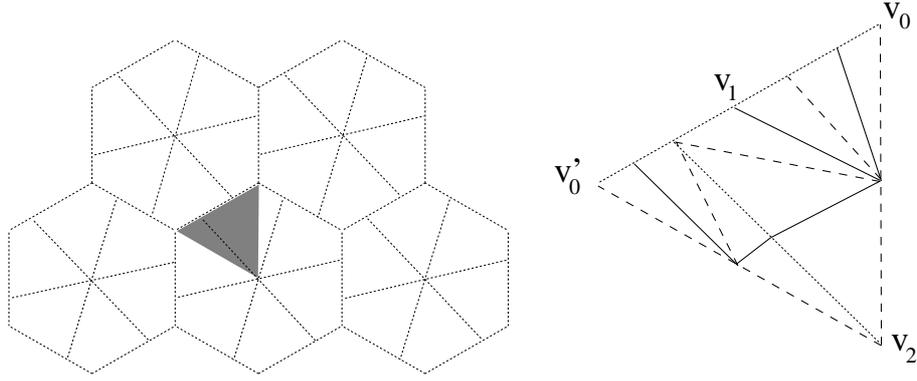}

\end{center}
\caption{A double chamber operation that cannot be described by separate subdivisions
of black and white chambers.}\label{fig:counterexample}
\end{figure}

Again, it must be shown that operations used in the literature are covered by this definition.
The mirror image plays a special role as an operation as it does not change the
combinatorial structure of a polyhedron. It is neither an lsp operation nor a 
lopsp operation. On the level of chamber systems it can be
described as changing the orientation of the chambers, so black ones become white
and the other way around.
We will show that {\em propellor}, {\em snub} and {\em whirl} 
can be obtained from tilings. Other operations (e.g. {\em gyro}) can be obtained by 
a combination of lsp operations and these lopsp operations. In Figure~\ref{fig:prop},Figure~\ref{fig:snub},
and  Figure~\ref{fig:whirl} the quadrangles $v_2,v_0,v_1,v_0'$  
are shown next to the tiling for better visibility. All chiral Caspar-Klug operations can be represented this way.

\begin{figure}
\begin{center}
\includegraphics[width=15cm]{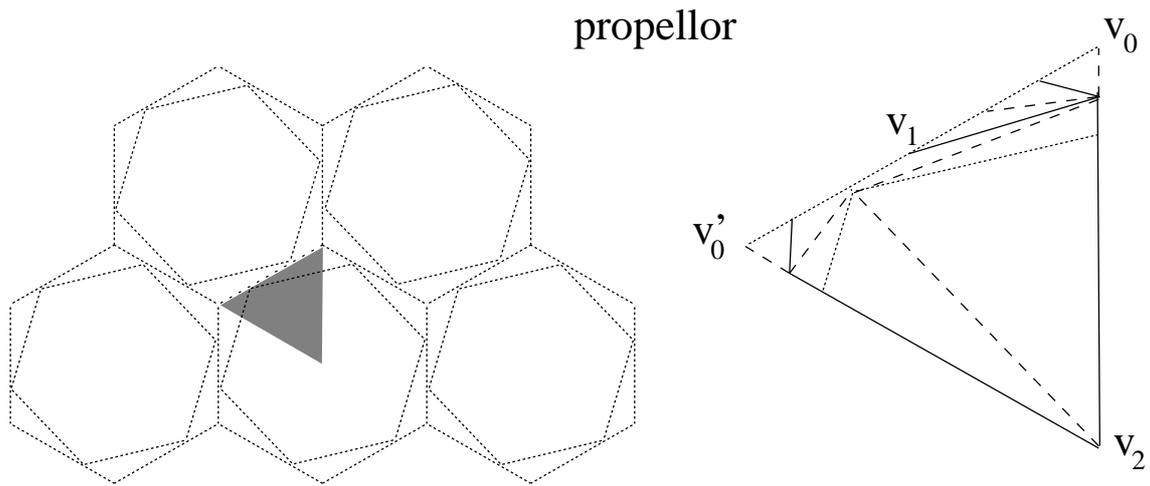}
\end{center}
\caption{ The lopsp operation {\em propellor}. }\label{fig:prop}
\end{figure}

\begin{figure}
\begin{center}
\includegraphics[width=15cm]{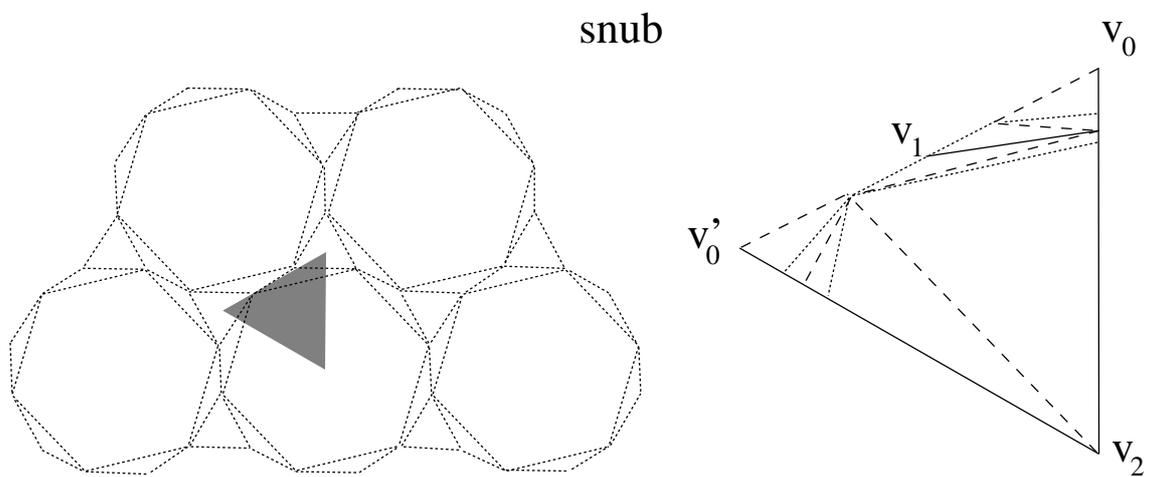}
\end{center}
\caption{ The lopsp operation {\em snub}. }\label{fig:snub}
\end{figure}

\begin{figure}
\begin{center}
\includegraphics[width=15cm]{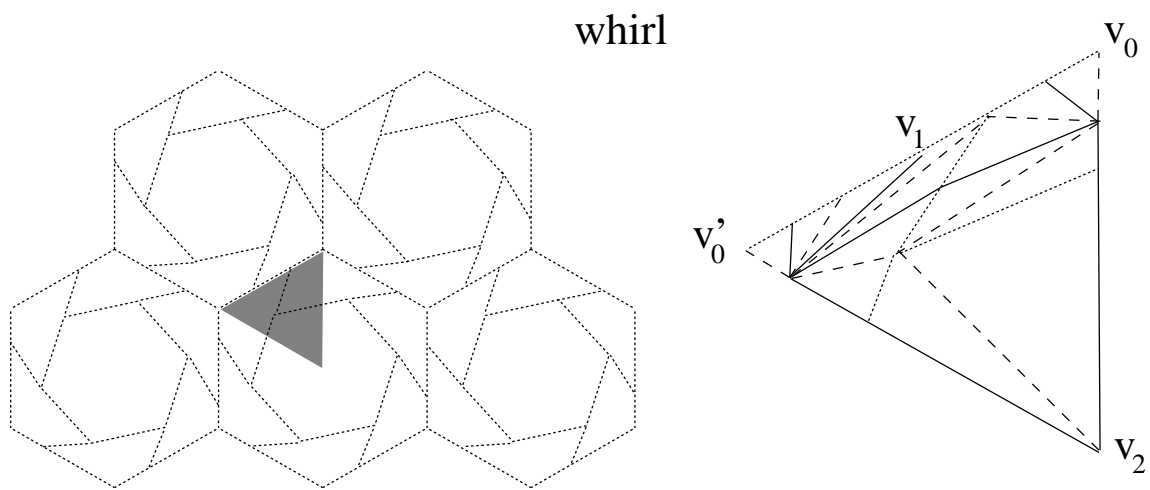}
\end{center}
\caption{ The lopsp operation {\em whirl}. }\label{fig:whirl}
\end{figure}

Although we could not find publications where decorations of double
chambers are used to represent lopsp operations like propellor, snub
or whirl, it may be considered common knowledge among people who work
with chamber systems that such a representation is possible. Again,
Definition~\ref{def:doublec} establishes the inverse direction and
gives a criterion for which decompositions of double chambers define an operation.

\section*{Conclusion and future work}

We have described the differences between the approaches of Goldberg,
Fuller and Caspar and Klug.  We have also made clear that what is
generally referred to as the Goldberg-Coxeter operation is due to
Caspar and Klug. Furthermore, we have explained the error in
Goldberg's paper and how it can be corrected.

By exactly defining -- based on Goldberg's approach -- local
symmetry-preserving operations (lsp) and local operations that preserve
orientation-preserving symmetries (lopsp), it becomes possible not only to study
specific operations but also to study the classes of all such
operations. These definitions make it possible to
design general proofs instead of examining each operation
individually (Theorem~\ref{thm:3con} is an example for this). 
The first task will now be to work out the details for lopsp operations, 
identifying which ways the boundary of the operation can be chosen when $v_2,v_0,v_1,v_0'$ are 
given, and proving that for each such choice the resulting decorated polyhedra are 
isomorphic. It is also necessary to prove a 
result corresponding to Theorem~\ref{thm:3con}.

The results presented in this paper also open new fields for investigation. Here are some
examples:

\begin{definition}

Two operations $O$, $O'$ are called equivalent if for all  polyhedra $P$ the results
of applying $O$ and $O'$ to $P$ are isomorphic. 

\end{definition}

While it is immediately evident that there are non-equivalent operations $O$, $O'$ that
produce isomorphic results for some polyhedra and non-isomorphic results
for others (see for example the operations {\em identity} and {\em dual} applied
to self-dual versus not self-dual polyhedra), it is not clear how exceptional
these cases are.
\medskip

{\bf Question:} 
 Are self-dual polyhedra the only polyhedra for which the application
of non-equivalent operations can give isomorphic results?
\medskip

This last question seems to be related to another one: lsp operations
preserve the symmetry group of the polyhedron to which they are
applied. In some cases they can also add new symmetries. An example is
the operation {\em ambo} applied to self-dual polyhedra. By studying
lsp operations applied to tilings of the torus, it is easy to see
that there exist operations that can increase symmetry but cannot be
written as a product of other operations. Consider for example a $4$-regular
toroidal tiling with $4$-gons where the group $Z_n\times Z_n$ is a
subgroup of the combinatorial symmetry group. If we apply the lsp
operation that subdivides each $4$-gon into $9$ smaller $4$-gons in a
$3\times 3$ pattern, we get $Z_{3n}\times Z_{3n}$ as a subgroup. The
number of chambers grows by a factor of $9$, while for all products
with {\em ambo} the numbers of chambers grow by an even factor -- so
the operation cannot be written as a product of ambo and other
operations.  Nevertheless, such examples are not known for polyhedra.

\medskip

{\bf Question:} 
Can all lsp operations that can increase the number of symmetries of a polyhedron
be written as a product of operations involving {\em ambo}?
\medskip

In order to measure the impact of an operation $O$ on the size or complexity of a  polyhedron, the ratios
of the numbers of vertices or faces before and after applying the operation are not good measures, 
as they do not depend on the operation alone. This can already be seen at the example of the
operation {\em dual}. Focussing on the edges of the polyhedra we get such an invariant:
the inflation rate $g(O)$, defined as the number of edges after applying the
operation to a polyhedron divided by the number of edges before the operation
is an invariant of the operation. 
It is equal to the number we would get by counting chambers before and after the operation
or by counting the number of chambers in an lsp operation $O$. For lopsp operations the inflation
rate is equal to half the number of chambers in the operation.
\medskip

{\bf Task:} Develop a computer program that can generate all
non-equivalent lsp and lopsp operations for a given inflation rate.
\medskip

Of course one could also ask for operations with a given inflation rate that produce only 
polyhedra with a given degree of the vertices or given face sizes.

Symmetric polyhedra and periodic tilings can be seen as containing redundant information,
as all the information is given by the structure of the fundamental domain and the symmetry group
distributing this information. This interpretation would not describe polyhedra with a trivial symmetry group
as redundant, but the situation is very similar to that of subdivisions of chambers in the role
of the fundamental domain and polyhedra in the role of the group:
when we apply operations with a large inflation factor to polyhedra with a trivial symmetry group,
in most cases the product will have a trivial symmetry group too and the information of the operation
-- the subdivided chamber -- is distributed via the chamber system of the polyhedron. 

\medskip

{\bf Task:} Develop an efficient algorithm that can determine whether a given polyhedron can be
described as an operation with inflation factor $g>1$ on a smaller polyhedron.
\medskip

\bibliographystyle{unsrt}
\bibliography{/home/gbrinkma/schreib/literatur}

\end{document}